\documentclass[12pt]{amsart}
\usepackage{amscd,amssymb}
\usepackage[all]{xy}

\renewcommand{\mod}{\operatorname{mod}\nolimits}

\newcommand{\Hom}{\operatorname{Hom}\nolimits}

\newcommand{\gr}{{\operatorname{gr}\nolimits}}
\newcommand{\rad}{\operatorname{rad}\nolimits}

\newcommand{\kar}{\operatorname{char}\nolimits}

\newcommand{\Ann}{\operatorname{Ann}\nolimits}

\newcommand{\Ext}{\operatorname{Ext}\nolimits}
\newcommand{\op}{{\operatorname{op}\nolimits}}

\newcommand{\MaxSpec}{\operatorname{MaxSpec}\nolimits}
\newcommand{\HH}{\operatorname{HH}\nolimits}

\newcommand{\gldim}{{\operatorname{gldim}\nolimits}}
\newcommand{\ev}{{\operatorname{ev}\nolimits}}

\newcommand{\m}{\mathfrak{m}}
\newcommand{\rrad}{\mathfrak{r}}
\newcommand{\mo}{\mathfrak{o}}
\newcommand{\mt}{\mathfrak{t}}
\newcommand{\oa}{\bar{a}}

\renewcommand{\L}{\Lambda}

\newcommand{\lan}{\Lambda_N}

\newcommand{\A}{{\mathcal A}}

\newcommand{\R}{{\mathcal R}}

\newcommand{\N}{{\mathcal N}}

\newtheorem{lem}{Lemma}[section]
\newtheorem{prop}[lem]{Proposition}

\newtheorem{thm}[lem]{Theorem}
\theoremstyle{definition}
\newtheorem{defin}[lem]{Definition}

\newtheorem{example}[lem]{Example}

\begin{document}

\title[Support varieties and the Hochschild cohomology ring \dots]
{Support varieties and the Hochschild cohomology ring modulo
nilpotence}
\author[Snashall]{Nicole Snashall}
\address{Nicole Snashall\\ Department of Mathematics\\
University of Leicester\\
University Road\\
Leicester, LE1 7RH\\
England}
\email{N.Snashall@mcs.le.ac.uk}

\begin{abstract}
This paper is based on my talks given at the `41st Symposium on Ring
Theory and Representation Theory' held at Shizuoka University,
Japan, 5-7 September 2008. It begins with a brief introduction to the use of
Hochschild cohomology in developing the theory of support varieties
of \cite{SS} for a module over an artin algebra.
I then describe the current status of research concerning the
structure of the Hochschild cohomology ring modulo nilpotence.
\end{abstract}

\thanks{This paper is in final form and no version of it will be submitted
for publication elsewhere.}

\maketitle

\section*{Acknowledgement}
I would like to thank the organisers of the `41st Symposium on Ring Theory
and Representation Theory' held at Shizuoka University for the invitation to
speak at this conference. The cost of visiting Japan and attending the
meeting in Shizuoka was supported by the Japan Society for Promotion of
Science (JSPS) Grant-in Aid for Scientific Research (B) 18340011. I would
particularly like to thank the head of the grant, Kiyoichi Oshiro (Yamaguchi
University), for providing this financial support and invitation. I would
also like to thank Shigeo Koshitani and Katsunori Sanada for their
hospitality during my visit to Japan. I also thank the University of
Leicester for granting me study leave.

\section*{Introduction}
This survey article is based on talks given at the `41st Symposium
on Ring Theory and Representation Theory', Shizuoka University in
September 2008, and is organised as follows. Section 1 gives a brief
introduction to the use of Hochschild cohomology in developing the
theory of support varieties of \cite{SS}. Section 2 considers the
Hochschild cohomology ring of $\Omega$-periodic algebras. In
\cite{SS} it had been conjectured that the Hochschild cohomology
ring modulo nilpotence of a finite-dimensional algebra is always
finitely generated as an algebra. Section 3 describes many classes
of algebras where this holds, that is, that the Hochschild
cohomology ring modulo nilpotence is finitely generated as an
algebra. The final section is devoted to studying the recent
counterexample of Xu (\cite{Xu}) to this conjecture.

Throughout this paper, let $\L$ be an indecomposable finite-dimensional algebra over an
algebraically closed field $K$, with Jacobson radical $\rrad$.
Denote by $\L^e$ the enveloping algebra
$\L^{\op}\otimes_K \L$ of
$\L$, so that right $\L^e$-modules correspond to $\L,\L$-bimodules.
The Hochschild cohomology ring
$\HH^*(\L)$ of $\L$ is given by $\HH^*(\L)= \Ext^*_{\L^e}(\L, \L) =
\oplus_{i\geq 0}\Ext^i_{\L^e}(\L,\L)$ with the Yoneda product. We
may consider an element of $\Ext_{\L^e}^n(\L, \L)$ as an exact
sequence of $\L,\L$-bimodules $0\to \L \to E^n \to E^{n-1} \to
\cdots \to E^1 \to \L \to 0$ where the Yoneda product is the
`splicing together' of exact sequences.

The low-dimensional Hochschild cohomology groups are well-understood
via the bar resolution (\cite{Ho} and see \cite{B, Ha}), and may be
described as follows:
\begin{enumerate}
\item[$\bullet$] $\HH^0(\Lambda) = Z(\L)$, the centre of $\L$.
\item[$\bullet$] $\HH^1(\L)$ is the space of derivations modulo the inner
derivations. A derivation is a $K$-linear map $f:\L \to \L$ such
that $f(ab) = af(b) + f(a)b$ for all $a, b \in \L$. A derivation $f
: \L \to \L$ is an inner derivation if there is some $x\in \L$ such
that $f(a) = ax-xa$ for all $a \in \L$.
\item[$\bullet$] $\HH^2(\L)$ measures the infinitesimal deformations of the
algebra $\L$; in particular, if $\HH^2(\L) = 0$ then $\L$ is rigid,
that is, $\L$ has no non-trivial deformations.
\end{enumerate}

Recently there has been much work on the structure of the entire
Hochschild cohomology ring $\HH^*(\L)$ and its connections and
applications to the representation theory of $\L$. One important
property of Hochschild cohomology in this situation is its
invariance under derived equivalence, proved by Rickard in
\cite[Proposition 2.5]{R2}
(see also \cite[Theorem 4.2]{Ha} for a special case).
It is also well-known that
$\HH^*(\L)$ is a graded commutative ring, that is, for homogeneous
elements $\eta \in \HH^n(\L)$ and $\theta \in \HH^m(\L)$, we have
$\eta\theta = (-1)^{mn}\theta\eta$. Thus, when the characteristic of
$K$ is different from two, then every homogeneous element of odd
degree squares to zero. Let $\N$ denote the ideal of $\HH^*(\L)$
which is generated by the homogeneous nilpotent elements. Then, for
$\kar K \neq 2$, we have $\HH^{2k+1}(\L) \subseteq \N$ for all $k
\geq 0$. Hence (in all characteristics) the Hochschild cohomology
ring modulo nilpotence, $\HH^*(\L)/\N$, is a commutative
$K$-algebra.

Support varieties for finitely generated modules over a
finite-dimensional algebra $\L$ were introduced using Hochschild
cohomology by Snashall and Solberg in \cite{SS}, where it was also
conjectured that the Hochschild cohomology ring modulo nilpotence is
itself a finitely generated algebra. We remark that the graded
commutativity of $\HH^*(\L)$ implies that $\N$ is contained in every
maximal ideal of $\HH^*(\L)$ and so $\MaxSpec \HH^*(\L) =
\MaxSpec\HH^*(\L)/\N$. Although the recent paper \cite{Xu} provides
a counterexample to the conjecture of \cite{SS}, nevertheless
finiteness conditions play an key role in the structure of these
support varieties (see \cite{EHSST}), so it remains of particular
importance to determine the structure of the Hochschild cohomology
ring modulo nilpotence.

\section{Support varieties}

One of the motivations for introducing support varieties for
finitely generated modules over a finite-dimensional algebra came
from the rich theory of support varieties for finitely generated
modules over group algebras of finite groups. For a finite group $G$
and finitely generated $KG$-module $M$, the variety of $M$,
$V_G(M)$, was defined by Carlson \cite{Ca} to be the variety of the
kernel of the homomorphism
$$-\otimes_K M \colon H^{\ev}(G, K) \rightarrow \Ext^*_{KG}(M, M).$$
This map factors through the Hochschild cohomology ring of $KG$, so
that we have the commutative diagram
$$\xymatrix@C=1pt{
H^{\ev}(G, K)\ar[rr]^{-\otimes_KM}\ar[dr] & & \Ext^*_{KG}(M, M)\\
& \HH^*(KG)\ar[ur]_{-\otimes_{KG}M} & }$$ Linckelmann considered the
map $-\otimes_{KG} M \colon \HH^{\ev}(KG) \rightarrow \Ext^*_{KG}(M,
M)$ when studying varieties for modules for non-principal blocks
(\cite{L}).

Now, for any finite-dimensional algebra $\L$ and finitely generated
$\L$-module $M$, there is a ring homomorphism $\HH^*(\L)
\stackrel{-\otimes_\L M}{\longrightarrow} \Ext_\L^*(M,M)$. This ring
homomorphism turns out to provide a similarly fruitful theory of
support varieties for finitely generated modules over an arbitrary
finite-dimensional algebra. As usual, let $\mod \L$ denote the
category of all finitely generated left $\L$-modules.

For $M \in \mod\L$, the support variety of $M$, $V_{\HH^*(\L)}(M)$,
was defined by Snashall and Solberg in \cite[Definition 3.3]{SS} by
$$V_{\HH^*(\L)}(M) = \{ \m\in \MaxSpec\HH^*(\L)/\N \mid
\Ann_{\HH^*(\L)}\Ext_\L^*(M,M) \subseteq \m'\}$$ where $\m'$ is the
preimage in $\HH^*(\L)$ of the ideal $\m$ in $\HH^*(\L)/\N$. We
recall from above that $\MaxSpec \HH^*(\L) = \MaxSpec\HH^*(\L)/\N$.

Since we assumed that $\L$ is indecomposable, we know that
$\HH^0(\L)$ is a local ring. Thus $\HH^*(\L)/\N$ has a unique
maximal graded ideal which we denote by $\m_{\gr}$ so that $\m_{\gr}
= \langle\rad\HH^0(\L),\HH^{\geq 1}(\L)\rangle/\N$. From \cite[Proposition
3.4(a)]{SS}, we have $\m_{\gr} \in V_{\HH^*(\L)}(M)$ for all $M \in
\mod \L$. We say that the variety of $M$ is {\it trivial} if
$V_{\HH^*(\L)}(M) = \{\m_{\gr}\}$.

The following result collects some of the properties of varieties
from \cite{SS}. For ease of notation, we write $V(M)$ for
$V_{\HH^*(\L)}(M)$. We also denote the kernel of the projective
cover of $M \in \mod\L$ by $\Omega_\L(M)$.

Recall that we assume throughout this paper that $K$ is an
algebraically closed field. This assumption is a necessary
assumption in many of the results in this article. However, it is
not needed in all of \cite{SS}, and the interested reader may refer
back to \cite{SS} to see precisely what assumptions are required
there at each stage.

\begin{thm}\label{prop:SS} (\cite[Propositions 3.4, 3.7]{SS}) Let $M \in \mod\L$.
\begin{enumerate}
\item[(1)] $V(M) = V(\Omega_\L(M))$ if $\Omega_\L(M) \neq (0)$,
\item[(2)] $V(M_1\oplus M_2) = V(M_1)\cup V(M_2)$,
\item[(3)] If $0\to M_1 \to M_2 \to M_3 \to 0$ is an exact sequence, then
$V(M_{i_1}) \subseteq V(M_{i_2})\cup V(M_{i_3})$ whenever $\{i_1,
i_2, i_3\} = \{1, 2, 3\}$,
\item[(4)] If $\Ext_\L^i(M, M) = (0)$ for $i \gg 0$, or the projective or the injective
dimension of $M$ is finite, then the variety of $M$ is trivial.
\item[(5)] If $\L$ is selfinjective then $V(M) = V(\tau M)$, where $\tau$ is the Auslander-Reiten
translate. Hence all modules in a connected stable component of the
Auslander-Reiten quiver have the same variety.
\end{enumerate}
\end{thm}

For a finitely generated module $M$ over a group algebra of a finite
group $G$, it is well-known (\cite{Ca}) that the variety
of $M$ is trivial if and only if $M$ is a projective module. In
contrast, it is still an open question as to what are the
appropriate necessary and sufficient conditions on a module for it
to have trivial variety in the more general case where $\L$ is an
arbitrary finite-dimensional algebra. There are some partial results
for a particular class of monomial algebras in \cite{FuSn} (see
Section \ref{HHmodN}). Nevertheless, the converse to Theorem
\ref{prop:SS}(4) does not hold in general, and a counterexample may
be found in \cite[Example 4.7]{S}.

However, this question was successfully answered by Erdmann,
Holloway, Snashall, Solberg and Taillefer in \cite{EHSST}, by
placing some (reasonable) additional assumptions on $\L$. (Recall
that we are already assuming that the field $K$ is algebraically
closed.) Specifically, the following two finiteness conditions were
introduced.

{\bf (Fg1)}\ $H$ is a commutative Noetherian graded subalgebra of
$\HH^*(\L)$ with $H^0 = \HH^0(\L)$.

{\bf (Fg2)}\ $\Ext^*_\L(\L/\rrad, \L/\rrad)$ is a finitely generated
$H$-module.

As remarked in \cite{EHSST}, these two conditions together imply
that both $\HH^*(\L)$ and $\Ext^*_\L(\L/\rrad, \L/\rrad)$ are
finitely generated $K$-algebras. In particular, the properties {\bf
(Fg1)} and {\bf (Fg2)} hold where $\L = KG$, $G$ is a finite group,
and $H = \HH^{\ev}(\L)$ (\cite{E,V}). With conditions {\bf (Fg1)}
and {\bf (Fg2)}, we have the following results from \cite{EHSST},
where we define the variety using the subalgebra $H$ of $\HH^*(\L)$,
so that $V_H(M) = \MaxSpec(H/\Ann_H\Ext_\L^*(M,M))$.

\begin{thm} (\cite[Theorem 2.5]{EHSST})
Suppose that $\Lambda$ and $H$ satisfy {\bf (Fg1)} and {\bf (Fg2)}.
Then $\Lambda$ is Gorenstein. Moreover the following are equivalent
for $M \in \mod \L$:
\begin{enumerate}
\item[(i)] The variety of $M$ is trivial;
\item[(ii)] $M$ has finite projective dimension;
\item[(iii)] $M$ has finite injective dimension.
\end{enumerate}
\end{thm}

\begin{thm} (\cite[Theorem 4.4]{EHSST})
Suppose that $\L$ and $H$ satisfy {\bf (Fg1)} and {\bf (Fg2)}. Given
a homogeneous ideal ${\mathfrak a}$ in $H$, there is a module $M \in
\mod \L$ such that $V_H(M) = V_H({\mathfrak a})$.
\end{thm}

\begin{thm} (\cite[Theorem 2.5 and Propositions 5.2, 5.3]{EHSST})
Suppose that $\L$ and $H$ satisfy {\bf (Fg1)} and {\bf (Fg2)} and
that $\Lambda$ is selfinjective. Let $M \in \mod \L$ be
indecomposable.
\begin{enumerate}
\item[(1)] $V_H(M)$ is trivial $\Leftrightarrow M$ is projective.
\item[(2)] $V_H(M)$ is a line $\Leftrightarrow M$ is $\Omega$-periodic.
\end{enumerate}
\end{thm}

Our final results in this section concern the representation type of
$\L$ and the structure of the Auslander-Reiten quiver; for more
details see \cite{EHSST,S}. First we recall that Heller showed that
if $\L$ is of
finite representation type then the complexity of a finitely
generated module is at most 1 (\cite{He}), and that Rickard showed that
if $\L$ is of
tame representation type then the complexity of a finitely generated
module is at most 2 (\cite{R}). However there are selfinjective
preprojective algebras of wild representation type where all
indecomposable modules are either projective or periodic and so have
complexity at most 1. Nevertheless, the next result uses the
Hochschild cohomology ring to give some information on the
representation type of an algebra.

\begin{thm} (\cite[Proposition 6.1]{EHSST})
Suppose that $\L$ and $H$ satisfy {\bf (Fg1)} and {\bf (Fg2)} and
that $\L$ is selfinjective. Suppose also that $\dim H \geq 2$. Then
$\L$ is of infinite representation type, and $\L$ has an infinite
number of indecomposable periodic modules lying in infinitely many
different components of the stable Auslander-Reiten quiver.
\end{thm}

We end this section with the statement of Webb's theorem (\cite{W})
for group algebras of finite groups and a generalisation of this
theorem from \cite{EHSST}.

\begin{thm} (\cite{W})
Let $G$ be a finite group and suppose that $\kar K$ divides $|G|$. Then
the orbit graph of a connected component of the stable
Auslander-Reiten quiver of $KG$ is one of the following:
\begin{enumerate}
\item[(a)] a finite Dynkin diagram $({\mathbb A}_n, {\mathbb D}_n, {\mathbb E}_{6,7,8})$,
\item[(b)] a Euclidean diagram $(\tilde{{\mathbb A}}_n, \tilde{{\mathbb
D}}_n, \tilde{{\mathbb E}}_{6,7,8}, \tilde{{\mathbb A}}_{12})$, or
\item[(c)] an infinite Dynkin diagram of
type ${\mathbb A}_\infty$, ${\mathbb D}_\infty$ or ${\mathbb
A}^\infty_\infty$.
\end{enumerate}
\end{thm}

\begin{thm}\label{thm:webb} (\cite[Theorem 5.6]{EHSST})
Suppose that $\Lambda$ and $H$ satisfy {\bf (Fg1)} and {\bf (Fg2)}
and that $\L$ is selfinjective. Suppose that the Nakayama functor is
of finite order on any indecomposable module in $\mod \L$. Then the
tree class of a component of the stable Auslander-Reiten quiver of
$\Lambda$ is one of the following:
\begin{enumerate}
\item[(a)] a finite Dynkin diagram $({\mathbb A}_n, {\mathbb D}_n, {\mathbb E}_{6,7,8})$,
\item[(b)] a Euclidean diagram $(\tilde{{\mathbb A}}_n, \tilde{{\mathbb
D}}_n, \tilde{{\mathbb E}}_{6,7,8}, \tilde{{\mathbb A}}_{12})$, or
\item[(c)] an infinite Dynkin diagram of
type ${\mathbb A}_\infty$, ${\mathbb D}_\infty$ or ${\mathbb
A}^\infty_\infty$.
\end{enumerate}
\end{thm}

We remark that the hypotheses of Theorem \ref{thm:webb} are
satisfied for all finite-dimensional cocommutative Hopf algebras
(\cite[Corollary 5.7]{EHSST}).

For more information, the reader should also see the survey paper on support
varieties for modules and complexes by Solberg \cite{S}. In addition, the
paper by Bergh \cite{Bergh} introduces the concept of a twisted
support variety for a finitely generated module over an artin algebra, where
the twist is induced by an automorphism of the algebra, and, in \cite{BS},
Bergh and Solberg study relative support varieties for finitely generated
modules over a finite-dimensional algebra over a field.

\section{$\Omega$-periodic algebras}

We now turn our attention to the structure of the Hochschild
cohomology ring. One class of algebras where it is relatively
straightforward to determine the structure of the Hochschild
cohomology ring explicitly is the class of $\Omega$-periodic
algebras. We recall that $\L$ is said to be an $\Omega$-periodic
algebra if there exists some $n \geq 1$ such that
$\Omega^n_{\L^e}(\L) \cong \L$ as bimodules. Such an algebra $\L$
has a periodic minimal projective bimodule resolution, so that
$\HH^i(\L) \cong \HH^{n+i}(\L)$ for $i \geq 1$, and is necessarily
self-injective (Butler; see \cite{GSS}).

There is an extensive survey of periodic algebras by Erdmann and
Skowro\'nski in \cite{ESk}. Examples of such algebras include
the preprojective algebras of
Dynkin type where $\Omega^6_{\L^e}(\L) \cong \L$ as bimodules
(\cite{RS}; see also \cite{ES}), and the deformed mesh
algebras of generalized Dynkin type of Bia\l kowski, Erdmann and
Skowro\'nski \cite{BES, ESk}.
For the selfinjective algebras of finite
representation type over an algebraically closed field, it is known
from \cite{GSS} that there is some $n \geq 1$ and automorphism
$\sigma$ of $\L$ such that $\Omega^n_{\L^e}(\L)$ is isomorphic as a
bimodule to the twisted bimodule ${}_1\L_{\sigma}$. It has now been
shown that all selfinjective algebras of finite
representation type over an algebraically closed field are
$\Omega$-periodic (\cite{D,EH,EHS,ESk}).

The structure of the Hochschild cohomology ring modulo nilpotence of
these algebras was determined by Green, Snashall and Solberg in \cite{GSS}.

\begin{thm}(\cite[Theorem 1.6]{GSS}) Let $K$ be an algebraically closed field.
Let $\L$ be a finite-dimensional indecomposable $K$-algebra such that there
is some $n \geq 1$ and some automorphism
$\sigma$ of $\L$ such that $\Omega^n_{\L^e}(\L)$ is isomorphic to the
twisted bimodule ${}_1\L_{\sigma}$. Then
$$\HH^*(\Lambda)/\N \cong \left \{ \begin{array}{ll}K[x] \mbox{ or } \\
K.\end{array}\right.$$ If there is some $m \geq 1$ such that
$\Omega^m_{\L^e}(\L) \cong \L$ as bimodules, then
$\HH^*(\Lambda)/\N \cong K[x]$, where $x$ is in degree $m$ and $m$ is
minimal.
\end{thm}

Additional information on the ring structure of the Hochschild
cohomology ring of the preprojective algebras of Dynkin type $A_n$
was determined in \cite{ES}, of the preprojective algebras of Dynkin
types $D_n, E_6, E_7, E_8$ in \cite{Eu}, and of the selfinjective
algebras of finite representation type $A_n$ over an algebraically
closed field in \cite{EH, EHS}.

\bigskip

Given that the Hochschild cohomology ring of these algebras is understood,
this naturally leads to the study of situations where the Hochschild
cohomology rings of two algebras $A$ and $B$ can be related. This enables us
to transfer information about the Hochschild cohomology ring of, say, an
$\Omega$-periodic algebra, to other algebras. Apart from periodic algebras,
there are other algebras where the Hochschild cohomology ring is known, and
these provide additional examples where the transfer of properties between
Hochschild cohomology rings may also be studied. One such class of examples
is the class of truncated quiver algebras, which has been extensively
studied in the literature by many authors.

Happel showed in \cite[Theorem 5.3]{Ha} that if $B$ is a one-point extension of
a finite-dimensional $K$-algebra $A$ by a finitely generated $A$-module $M$,
then there is a long exact sequence connecting the Hochschild cohomology
rings of $A$ and $B$:
$$0\to\HH^0(B) \to \HH^0(A) \to \Hom_A(M,M)/K \to$$
$$\HH^1(B) \to \HH^1(A) \to \Ext^1_A(M,M)\to \cdots$$
$$\cdots \to \Ext^{i}_A(M,M) \to \HH^{i+1}(B) \to \HH^{i+1}(A) \to
\Ext^{i+1}_A(M,M) \to \cdots$$
It was subsequently shown by Green, Marcos and Snashall in
\cite[Theorem 5.1]{GMS} that there is a graded ring homomorphism
$$\HH^*(B) \to \HH^*(A)\oplus {\mathcal K}$$
which induces this long exact sequence,
where
${\mathcal K}$ is the graded $K$-module with ${\mathcal K}_0 = K$ and
${\mathcal K}_n = 0$ for all $n \neq 0$.

These results were generalized independently to arbitrary triangular matrix
algebras by Cibils \cite{C}, by Green and Solberg \cite{GSo}, and by
Michelena and Platzeck \cite{MP}.

\bigskip

A recent result of K\"onig and Nagase, (\cite{KN, N}), has related the
Hochschild cohomology ring of $B$ to that of $B/BeB$ in the case where $B$ is
an algebra with idempotent $e$, such that $BeB$ is a stratifying
ideal of $B$.

\begin{thm}(\cite{KN})
Let $B$ be an algebra with idempotent $e$ such that $BeB$ is a
stratifying ideal of $B$ and let $A$ be the factor algebra $B/BeB$.
Then there are long exact sequences as follows:
\begin{enumerate}
\item $\cdots\rightarrow \Ext^n_{B^e}(B,BeB)
\rightarrow \HH^n(B) \rightarrow \HH^n(A) \rightarrow\cdots ;$
\item $\cdots\rightarrow \Ext^n_{B^e}(A,B)
\rightarrow \HH^n(B) \rightarrow \HH^n(eBe) \rightarrow\cdots ; and$
\item $\cdots\rightarrow \Ext^n_{B^e}(A,BeB)
\rightarrow \HH^n(B) \rightarrow \HH^n(A) \oplus
\HH^n(eBe)\rightarrow\cdots .$
\end{enumerate}
\end{thm}

\section{The Hochschild cohomology ring modulo nilpotence}\label{HHmodN}

The definition of a support variety in \cite{SS} led us to
consider the structure of $\HH^*(\L)/\N$ and to conjecture that
$\HH^*(\L)/\N$ is always finitely generated as an
algebra. A counterexample to this conjecture was recently given by Xu in
\cite{Xu}; nevertheless the Hochschild cohomology ring modulo
nilpotence is finitely generated as an algebra for many diverse
classes of algebras.

The Hochschild cohomology ring modulo nilpotence is known to be
finitely generated as an algebra in the following cases.
\begin{enumerate}
\item[$\bullet$] any block of a group ring of a finite group
(\cite{E,V});
\item[$\bullet$] any block of a finite-dimensional cocommutative Hopf
algebra (\cite{FS});
\item[$\bullet$] finite-dimensional selfinjective algebras of finite
representation type over an algebraically closed field (\cite{GSS});
\item[$\bullet$] finite-dimensional monomial algebras
(\cite{GSSmon} and see \cite{GS});
\item[$\bullet$] finite-dimensional algebras of finite global
dimension (see \cite{Ha}).
\end{enumerate}

For the last class of examples, if $\L$ is an algebra of finite
global dimension $N$, then $\HH^i(\L) = \Ext^i_{\L^e}(\L, \L) = (0)$
for all $i > N$. Hence $\HH^*(\L)/\N \cong K$. In \cite{Ha}, Happel
asked whether or not it was true, for a finite-dimensional algebra
$\Gamma$ over a field $K$, that if $\HH^n(\Gamma) = (0)$ for $n \gg
0$ then the global dimension of $\Gamma$ is finite. This question
has now been answered in the negative by Buchweitz, Green, Madsen and
Solberg in \cite{BGMS} by the
following example.

\begin{example} (\cite{BGMS})
Let $$\L_q = K\langle x, y\rangle/(x^2, xy+qyx, y^2)$$ with $q \in
K\setminus\{0\}$. If $q$ is not a root of unity then
$\dim\HH^i(\L_q) = 0$ for $i \geq 3$. Moreover, $\L_q$ is a
selfinjective algebra so has infinite global dimension.
\end{example}

We also note from \cite{BGMS} that $\dim\L_q = 4$, $\dim\HH^*(\L_q)
= 5$ and $\HH^*(\L)/\N \cong K$.

\bigskip

However, the situation for commutative algebras is very different, as
Avramov and Iyengar have shown.

\begin{thm} (\cite{AI})
Let $R$ be a commutative finite-dimensional $K$-algebra over a field
$K$. If $\HH^n(R) = (0)$ for $n \gg 0$ then $R$ is a (finite)
product of (finite) separable field extensions of $K$. In
particular, the global dimension of $R$ is finite.
\end{thm}

We now turn to a brief discussion of the
Hochschild cohomology ring modulo nilpotence for a monomial algebra, which
was studied by Green, Snashall and Solberg.
Let $\L$ be a quotient of a path algebra so that $\L=K\mathcal{Q}/I$
for some quiver ${\mathcal Q}$ and admissible ideal $I$ of
$K\mathcal{Q}$. Then $\L=K\mathcal{Q}/I$ is a monomial algebra if
the ideal $I$ is generated by monomials of length at least two. It
should be noted that monomial algebras are very rarely selfinjective
and so do not usually exhibit the same properties as group algebras.
However, the Hochschild cohomology ring modulo nilpotence of a
monomial algebra turns out to have a particularly nice structure.

\begin{thm} (\cite[Theorem 7.1]{GSSmon})
Let $\L = K\mathcal{Q}/I$ be a finite-dimensional indecomposable
monomial algebra.
Then $\HH^*(\L)/\N$ is a commutative finitely generated $K$-algebra
of Krull dimension at most one.
\end{thm}

For some specific subclasses of monomial algebras, the structure of
the Hochschild cohomology ring modulo nilpotence was explicitly
determined in \cite{GS}. One of the main tools used was the minimal
projective bimodule resolution of a monomial algebra of Bardzell
\cite{Ba}. In order to define the particular class of $(D,A)$-stacked
monomial algebras, we require the concept of overlaps of \cite{GHZ,
GZ}; the definitions here use the notation of \cite{GS}.

\begin{defin}
A path $q$ in $K{\mathcal Q}$ overlaps a path $p$ in $K{\mathcal Q}$
with overlap $pu$ if there are paths
$u$ and $v$ such that $pu = vq$ and $1 \leq \ell(u) < \ell(q)$, where
$\ell(x)$ denotes the length of the path $x\in K{\mathcal Q}$. We
illustrate the definition with the following diagram. (Note that we
allow $\ell(v) = 0$ here.)
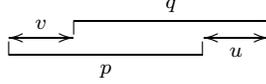
\[\xymatrix@W=0pt@M=0.3pt{
\ar@{^{|}-^{|}}@<-1.25ex>[rrr]_p\ar@{{<}-{>}}[r]^{v} &
\ar@{_{|}-_{|}}@<1.25ex>[rrr]^q & & \ar@{{<}-{>}}[r]_{u} & & }\]

A path $q$ properly overlaps a path $p$ with overlap $pu$ if $q$
overlaps $p$ and $\ell(v) \geq 1$.
\end{defin}

Let $\L = K{\mathcal Q}/I$ be a finite-dimensional monomial algebra
where $I$ has a minimal set of generators $\rho$ of paths of length
at least 2. We fix this set $\rho$ and now recursively define sets
$\R^n$ (contained in $K{\mathcal Q}$). Let
$$\begin{array}{lll}
\R^0 & = & \mbox {the set of vertices of ${\mathcal Q}$,}\\
\R^1 & = & \mbox{the set of arrows of ${\mathcal Q}$,}\\
\R^2 & = & \mbox{$\rho$.}
\end{array}$$
For $n \geq 3$, we say $R^2 \in \R^2$ maximally overlaps $R^{n-1}
\in \R^{n-1}$ with overlap $R^n = R^{n-1}u$ if
\begin{enumerate}
\item $R^{n-1} = R^{n-2}p$ for some path $p$;
\item $R^2$ overlaps $p$ with overlap $pu$;
\item there is no element of $\R^2$ which overlaps $p$ with overlap being a
proper prefix of $pu$.
\end{enumerate}
The set $\R^n$ is defined to be the set of all overlaps $R^n$ formed
in this way.

\bigskip

Each of the elements in $\R^n$ is a path in the quiver ${\mathcal
Q}$. We follow the convention that paths are written from left to
right. For an arrow $\alpha$ in the quiver ${\mathcal Q}$, we write
$\mo(\alpha)$ for the idempotent corresponding to the origin of
$\alpha$ and $\mt(\alpha)$ for the idempotent corresponding to the
tail of $\alpha$ so that $\alpha = \mo(\alpha)\alpha\mt(\alpha)$.
For a path $p = \alpha_1\alpha_2\cdots \alpha_m$, we write $\mo(p) =
\mo(\alpha_1)$ and $\mt(p) = \mt(\alpha_m)$.

The importance of these sets $\R^n$ lies in the fact that, for a
finite-dimensional monomial algebra $\L = K{\mathcal Q}/I$,
\cite{Ba} uses them to give an explicit construction of a minimal projective
bimodule resolution $(P^*, \delta^*)$ of $\L$, showing that
$$P^n = \oplus_{R^n \in\R^n}\L\mo(R^n)\otimes_K\mt(R^n)\L.$$

We now use these same sets $\R^n$ to define a $(D,A)$-stacked
monomial algebra.

\begin{defin}\label{definstacked}(\cite[Definition 3.1]{GS})
Let $\L = K\mathcal{Q}/I$ be a finite-dimensional monomial algebra,
where $I$ is an admissible ideal with minimal set of generators
$\rho$. Then $\L$ is said to be a $(D,A)$-stacked monomial algebra
if there is some $D \geq 2$ and $A \geq 1$ such that, for all $n
\geq 2$ and $R^n \in \R^n$,
$$\ell(R^n) = \left \{ \begin{array}{ll}
\frac{n}{2}D & \mbox{if $n$ even,}\\\\
\frac{(n-1)}{2}D + A & \mbox{if $n$ odd.}
\end{array} \right.$$
In particular all relations in $\rho$ are of length $D$.
\end{defin}

The class of $(D,A)$-stacked monomial algebras includes the Koszul
monomial algebras (equivalently, the quadratic monomial algebras)
and the $D$-Koszul monomial algebras of Berger (\cite{Ber}). Recall
that the Ext algebra $E(\L)$ of $\L$ is defined by $E(\L) =
\Ext^*_{\L}(\L/\rrad, \L/\rrad)$. It is well-known that the Ext
algebra of a Koszul algebra is generated in degrees 0 and 1;
moreover the Ext algebra of a $D$-Koszul algebra is generated in
degrees 0, 1 and 2 (\cite{GMMVZ}). It was shown by Green and Snashall in
\cite[Theorem 3.6]{GSext} that, for algebras of infinite
global dimension, the class of $(D,A)$-stacked monomial algebras is
precisely the class of monomial algebras $\L$ where each projective
module in the minimal projective resolution of $\L/\rrad$ as a right
$\L$-module is generated in a single degree and where the Ext
algebra of $\L$ is finitely generated as a $K$-algebra. It
was also shown,
for a $(D,A)$-stacked monomial algebra of infinite global dimension,
that the Ext algebra is generated in degrees 0, 1, 2 and 3.

\begin{thm}\cite{GS}
Let $\L = K\mathcal{Q}/I$ be a finite-dimensional $(D,A)$-stacked
monomial algebra, where $I$ is an admissible ideal with minimal set
of generators $\rho$. Suppose $\kar K \neq 2$ and $\gldim\L \geq 4$.
Then there is some integer $r \geq 0$ such that
$$\HH^*(\L)/\N \cong K[x_1, \ldots , x_r]/\langle x_ix_j \mbox{ for } i \neq
j\rangle.$$ Moreover the degrees of the $x_i$ and the value of the
parameter $r$ may be explicitly and easily calculated.
\end{thm}

We do not give the full details of the $x_i$ and the parameter $r$
here; they may be found in \cite{GS}. However, it is worth remarking
that, given any integer $r\geq 0$ and even integers $n_1, \ldots ,
n_r$, there is a finite-dimensional $(D,A)$-stacked monomial algebra
$\L$ with $$\HH^*(\L)/\N \cong K[x_1, \ldots , x_r]/\langle x_ix_j
\mbox{ for } i \neq j\rangle$$ where the degree of $x_i$ is $n_i$,
for all $i = 1, \ldots , r$.

In \cite{FuSn}, necessary and sufficient conditions are given for a simple
module over a $(D,A)$-stacked monomial algebra to have trivial variety.
Referring back to Theorem \ref{prop:SS}(4), this goes part way to
determining necessary and sufficient conditions on any finitely generated
module for it to have trivial variety for this class of algebras.

\bigskip

We end this section with a class of selfinjective special biserial
algebras $\lan$, for $N \geq 1$, studied by Snashall and Taillefer
in \cite{ST}. The study of
these algebras was motivated by the results of \cite{EGST} where the
algebras $\L_1$ arose in the presentation by quiver and relations of
the Drinfeld double ${\mathcal D}(\Lambda_{n,d})$ of the Hopf
algebra $\Lambda_{n,d}$ where $d|n$. The algebra $\Lambda_{n,d}$ is
given by an oriented cycle with $n$ vertices such that all paths of
length $d$ are zero. These algebras also occur in the study of the
representation theory of $U_q(\mathfrak{sl}_2)$; see work of Patra
(\cite{P}), Suter (\cite{Su}), Xiao (\cite{X}),
and also of Chin and Krop (\cite{CK}). The more general algebras $\L_N$
occur in work of Farnsteiner and Skowro\'nski
\cite{FaSk, FaSk2}, where they determine the Hopf algebras associated to
infinitesimal groups whose principal blocks are tame when $K$ is an
algebraically closed field with $\kar K \geq 3$.

Our class of selfinjective special biserial algebras $\lan$ is
described as follows. Firstly, for $m \geq 1$, let ${\mathcal Q}$ be
the quiver with $m$ vertices, labelled $0, 1 \ldots, m-1$, and $2m$
arrows as follows:
$$
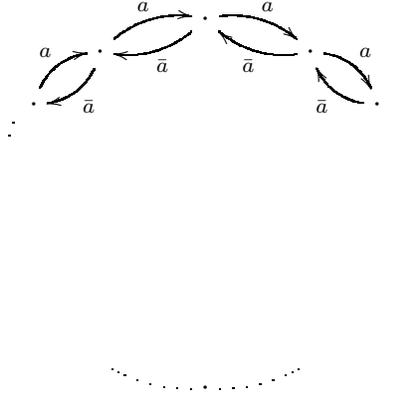
\xymatrix@=.01cm{
&&&&&&&&&&&\cdot\ar@/^.5pc/[rrrrrd]^{a}\ar@/^.5pc/[llllld]^{\bar{a}}\\
&&&&&&\cdot\ar@/^.5pc/[rrrrru]^{a}\ar@/^.5pc/[llldd]^{\bar{a}}&&&&&&&&&&
\cdot\ar@/^.5pc/[rrrdd]^{a}\ar@/^.5pc/[lllllu]^{\bar{a}}\\\\
&&&\cdot\ar@/^.5pc/[rrruu]^{a}\ar@{.}@/_.3pc/[ldd]&&&&&&&&&&&&&&&&
\cdot\ar@/^.5pc/[llluu]^{\bar{a}}\ar@{.}@/^.3pc/[rdd]\\\\
&&&&&&&&&&&&&&&&&&&&&&&\\
\\
\\
\\\\\\\\\\\\\\\\
&&&&&&&&&&&&&&&&\\
&&&&&&&&&&&\cdot\ar@/_.3pc/@{.}[rrrrru]\ar@/^.3pc/@{.}[lllllu] }$$
Let $a_i$ denote the arrow that goes from vertex $i$ to vertex
$i+1$, and let $\oa_i$ denote the arrow that goes from vertex $i+1$
to vertex $i$, for each $i=0, \ldots , m-1$ (with the obvious
conventions modulo $m$). Then, for $N \geq 1$, we define $\lan$ to
be the algebra given by $\lan=K{\mathcal Q}/I_N$ where $I_N$ is the
ideal of $K{\mathcal Q}$ generated by
$$a_ia_{i+1},\ \ \oa_{i-1}\oa_{i-2},\ \
(a_i\oa_i)^{N}-(\oa_{i-1}a_{i-1})^{N}, \ \ \mbox{ for } i=0, 1,
\ldots , m-1,$$ and where the subscripts are taken modulo $m$. We
note that, if $N = 1$, then the algebra $\Lambda_1$ is a Koszul
algebra. (We continue to write paths from left to right.)

\begin{thm} (\cite[Theorem 8.1]{ST})
For $m \geq 1$ and $N \geq 1$, let $\lan$ be as defined above. Then
$\HH^*(\lan)$ is a finitely generated $K$-algebra. Moreover
$\HH^*(\lan)/\N$ is a commutative finitely generated $K$-algebra of
Krull dimension two.
\end{thm}

Furthermore, if $N = 1$ then \cite{ST} also showed that the
conditions {\bf (Fg1)} and {\bf (Fg2)} hold with $H =
\HH^{\ev}(\L_1)$.

\section{Counterexample to the conjecture of \cite{SS}}

The previous section concerned algebras where the conjecture of
\cite{SS} concerning the finite generation of the Hochschild
cohomology ring modulo nilpotence has been shown to hold. In this
section we present a counterexample to the conjecture of \cite{SS}.
In \cite{Xu}, Xu gave a counterexample in the case where the field
$K$ has characteristic 2. It can easily be seen, for $\kar K = 2$,
that the category algebra he presented in \cite{Xu}
is isomorphic to the following algebra $\A$ given as a quotient of a
path algebra. Moreover, we will show that this algebra $\A$ provides
a counterexample to the conjecture irrespective of the
characteristic of the field.

\begin{example}\label{counterex}
Let $K$ be any field and let $\A = K{\mathcal Q}/I$ where ${\mathcal
Q}$ is the quiver
$$
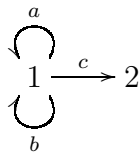
\xymatrix{
1\ar@(ur,ul)[]_{a}\ar@(dr,dl)[]^{b}\ar[r]^c & 2 }$$ and $I = \langle
a^2, b^2, ab-ba, ac\rangle$.
\end{example}

The rest of this section is devoted to studying this algebra $\A$
and to showing that $\HH^*(\A)/\N$ is not finitely generated as an
algebra.

We begin by giving an explicit minimal projective resolution $(P^*,
d^*)$ for $\A$ as an $\A,\A$-bimodule. The description of the
resolution given here is motivated by \cite{GSres} where the first
terms of a minimal projective bimodule resolution of a finite-dimensional
quotient of a path algebra were determined explicitly from the
minimal projective resolution of $\Lambda/\rrad$ as a right $\L$-module
of Green, Solberg and Zacharia in \cite{GSZ}. This same technique for
constructing a minimal
projective bimodule resolution was used in \cite{GHMS} for any
Koszul algebra, and in \cite{ST} for the algebras $\lan$ which were
discussed at the end of Section \ref{HHmodN}.

From Happel \cite{Ha}, we know that the multiplicity of $\Lambda e_i\otimes_K
e_j\Lambda$ as a direct summand of $P^n$ is equal to
$\dim\Ext^n_\Lambda(S_i, S_j)$, where $S_i$ is the simple
$\A$-module corresponding to the vertex $i$ of ${\mathcal Q}$.
Following \cite{GSres,GSZ}, we start by defining sets $g^n$ in $K{\mathcal
Q}$ inductively, and then labelling the summands of
$P^n$ by the elements of $g^n$. The set $g^0$ is determined by the vertices
of ${\mathcal Q}$, the set $g^1$ by the arrows of ${\mathcal Q}$, and the
set $g^2$ by a minimal generating set of the ideal $I$.

Let \begin{align*} g^0 &= \{g^0_0 = e_1,\ g^0_1 = e_2\},\\
g^1 &= \{g^1_0 = a,\ g^1_1 = -b,\ g^1_2 = c\},\\
g^2 &= \{g^2_0 = a^2,\ g^2_1 = ab-ba,\ g^2_2 = -b^2,\ g^2_3 = ac\}.
\end{align*}
For $n \geq 3$ and $r = 0, 1, \ldots ,  n$, let $$g^n_r = \sum_{p}
(-1)^sp$$ where the sum is over all paths $p$ of length $n$, written
$p= \alpha_1 \alpha_2 \cdots \alpha_n$ where the $\alpha_i$ are
arrows in ${\mathcal Q}$, such that
\begin{enumerate}
\item[(i)] $p$ contains $n-r$ arrows equal to $a$ and $r$ arrows equal to
$b$, and
\item[(ii)] $s = \sum_{\alpha_j = b}j$.
\end{enumerate}
In addition, for $n \geq 3$, define $$g^n_{n+1} = a^{n-1}c.$$

For $r = 0, 1, \ldots , n$, we have that $g^n_r = e_1g^n_re_1$ so we
define $\mo(g^n_r) = e_1 = \mt(g^n_r)$. Moreover $\mo(g^n_{n+1}) =
e_1$ and $\mt(g^n_{n+1}) = e_2$. Thus $$P^n =
\oplus_{r=0}^{n+1}\A\mo(g^n_r) \otimes_K \mt(g^n_r)\A.$$

\bigskip

To describe the map $d^n\colon P^n\to P^{n-1}$, we first need to write
each of the elements $g^n_r$ in terms of the elements of the set $g^{n-1}$,
that is, in terms of $g^{n-1}_0, \ldots ,
g^{n-1}_n$. The following result is straightforward to verify.

\begin{prop}
Suppose $n \geq 2$. Then, keeping the above notation,
$$\begin{array}{llll} g^n_0 &= g^{n-1}_0a &= ag^{n-1}_0 & \\
g^n_r &= g^{n-1}_ra + (-1)^ng^{n-1}_{r-1}b &= (-1)^r(ag^{n-1}_r + bg^{n-1}_{r-1}) & \mbox{ for } 1 \leq r \leq n-1\\
g^n_n &= (-1)^ng^{n-1}_{n-1}b &= (-1)^nbg^{n-1}_{n-1} &\\
g^n_{n+1} &= g^{n-1}_0c &= ag^{n-1}_n. &\\
\end{array}$$
\end{prop}

We define the map $d^0\colon P^0\to \A$ to be the multiplication
map. To define $d^n$ for $n \geq 1$, we need one further piece of
notation. In describing the image $d^n(\mo(g^n_r) \otimes
\mt(g^n_r))$ in the projective module $P^{n-1}$, we use a subscript
under $\otimes$ to indicate the appropriate summand of the
projective module $P^{n-1}$. Specifically, let $-\otimes_r-$ denote
a term in the summand of $P^{n-1}$ corresponding to $g^{n-1}_r$.
Nonetheless all tensors are over $K$; however, to simplify notation,
we omit the subscript $K$. The maps $d^n\colon P^n\to P^{n-1}$ for
$n \geq 1$ may now be defined. The proof of the following result is
omitted and is similar to those in \cite[Proposition 2.8]{GSres} and
\cite[Theorem 1.6]{ST}.

\begin{thm}\label{thm:resol}
For the algebra $\A$ of Example \ref{counterex}, the sequence
$(P^*,d^*)$ is a minimal projective resolution of $\A$ as an
$\A,\A$-bimodule, where, for $n \geq 0$,
$$P^n = \bigoplus_{r=0}^{n+1}\A\mo(g^n_r)\otimes\mt(g^n_r)\A,$$
the map $d^0 : P^0 \to \A$ is the multiplication map,
the map $d^1 : P^1 \to P^0$ is given by
$d^1(\mo(g^1_r)\otimes\mt(g^1_r)) =$
$$\left\{ \begin{array}{ll}
\mo(g^1_0)\otimes_{0}a - a\otimes_{0}\mt(g^1_0) & \mbox{ for } r=0;\\
-\mo(g^1_1)\otimes_{0}b + b\otimes_{0}\mt(g^1_1) & \mbox{ for } r=1;\\
\mo(g^1_{2})\otimes_{0}c - c\otimes_{1}\mt(g^1_{2})
& \mbox{ for } r=2;
\end{array}\right.$$
and, for $n \geq 2$,  the map $d^n : P^n \to P^{n-1}$ is given by
$d^n(\mo(g^n_r)\otimes\mt(g^n_r)) =$
$$\left\{ \begin{array}{ll}
\mo(g^n_0)\otimes_{0}a + (-1)^n
a\otimes_{0}\mt(g^n_0) & \mbox{ for } r=0;\\
\mo(g^n_r)\otimes_{r}a + (-1)^n\mo(g^n_r)\otimes_{r-1}b & \\
\hspace{2cm} + (-1)^{r+n}(a\otimes_{r}\mt(g^n_r) +
b\otimes_{r-1}\mt(g^n_r))
& \mbox{ for } 1 \leq r \leq n-1;\\
(-1)^n\mo(g^n_n)\otimes_{n-1}b +
b\otimes_{n-1}\mt(g^n_n) & \mbox{ for } r = n;\\
\mo(g^n_{n+1})\otimes_{0}c + (-1)^na\otimes_{n}\mt(g^n_{n+1})
& \mbox{ for } r=n+1.\\
\end{array}\right.$$
\end{thm}

\begin{prop} The algebra $\A$ of Example \ref{counterex} is a Koszul
algebra.
\end{prop}

\begin{proof}
We apply the functor $\A/\rrad \otimes_{\A} -$ to the resolution
$(P^*,d^*)$ of Theorem \ref{thm:resol} to give a (minimal)
projective resolution of $\A/\rrad$ as a right $\A$-module. Thus
$\A/\rrad$ has a linear projective resolution, and so $\A$ is a
Koszul algebra.
\end{proof}

Since, $\A$ is a Koszul algebra, it now follows from \cite{BGSS}
that the image of the ring homomorphism $\phi_{\A/\rrad} =
\A/\rrad\otimes_\A- : \HH^*(\A)
\to E(\A)$ is the graded centre $Z_\gr(E(\A))$ of $E(\A)$, where
$Z_\gr(E(\A))$ is the subalgebra generated by all homogeneous
elements $z$ such that $zg = (-1)^{|g||z|}gz$ for all $g \in E(\A)$.
Thus $\phi_{\A/\rrad}$ induces an isomorphism $\HH^*(\A)/\N \cong
Z_\gr(E(\A))/\N_Z$, where $\N_Z$ denotes the ideal of $Z_\gr(E(\A))$
generated by all homogeneous nilpotent elements.

From \cite[Theorem 2.2]{GMV}, $E(\A)$
is the Koszul dual of $\A$ and is given explicitly by quiver and relations
as $E(\A) \cong K{\mathcal Q}^{\op}/I^\perp$, where
${\mathcal Q}$ is the quiver of $\A$ and $I^\perp$ is the ideal generated by
the orthogonal relations to those of $I$. Specifically, for this example,
$E(\A)$ has quiver
$$\xymatrix{
1\ar@(ur,ul)[]_{a^o}\ar@(dr,dl)[]^{b^o} & 2\ar[l]^{c^o} }$$
and $I^{\perp} = \langle a^ob^o + b^oa^o, b^oc^o\rangle$, where, for an
arrow $\alpha \in {\mathcal Q}$, we denote by $\alpha^o$ the corresponding
arrow in ${\mathcal Q}^{\op}$.
Moreover, the left modules over $E(\A)$ are the right modules over
$K{\mathcal Q}/\langle ab + ba, bc\rangle$.

It is now
easy to calculate $Z_\gr(E(\A))$ to give the following theorem. The
structure of $\HH^*(\A)/\N$ for $\kar K = 2$ was given by Xu in
\cite{Xu}.

\begin{thm}
Let $\A$ be the algebra of Example \ref{counterex}.
\begin{enumerate}
\item $$Z_\gr(E(\A)) \cong \left \{ \begin{array}{ll}
K\oplus K[a,b]b & \mbox{ if } \kar K = 2\\
K\oplus K[a^2,b^2]b^2 & \mbox{ if } \kar K \neq
2,\end{array}\right.$$ where $b$ is in degree 1 and $ab$ is in
degree 2.
\item $$\HH^*(\A)/\N \cong \left \{ \begin{array}{ll}
K\oplus K[a,b]b & \mbox{ if } \kar K = 2\\
K\oplus K[a^2,b^2]b^2 & \mbox{ if } \kar K \neq
2,\end{array}\right.$$ where $b$ is in degree 1 and $ab$ is in
degree 2.
\item $\HH^*(\A)/\N$ is not finitely generated as an algebra.
\end{enumerate}
\end{thm}

This example now raises the new question as to whether we can give
necessary and sufficient conditions on a finite-dimensional algebra
for its Hochschild cohomology ring modulo nilpotence to be finitely
generated as an algebra.

\end{document}